\documentclass[conference]{IEEEtran}
\IEEEoverridecommandlockouts
\usepackage{cite}
\usepackage{amsmath,amssymb,amsfonts, amsthm}
\usepackage{algorithmic}
\usepackage{graphicx}
\usepackage{textcomp}
\usepackage{xcolor}
\usepackage{textcomp}
\def\BibTeX{{\rm B\kern-.05em{\sc i\kern-.025em b}\kern-.08em
    T\kern-.1667em\lower.7ex\hbox{E}\kern-.125emX}}
   
\newtheorem{theorem}{Theorem}
\theoremstyle{plain}

\begin{document}

\title{Control of Two Energy Storage Units with Market Impact: Lagrangian Approach and Horizons\\
{\footnotesize }
\thanks{Albert Sol\`a Vilalta was supported by The Maxwell Institute Graduate School in Analysis and its Applications, a Centre for Doctoral Training funded by the UK Engineering and Physical Sciences Research Council (grant EP/L016508/01), the Scottish Funding Council, Heriot-Watt University and the University of Edinburgh.}
}

\author{\IEEEauthorblockN{Miguel F. Anjos}
\IEEEauthorblockA{\textit{Maxwell Institute for Mathematical Sciences} \\
\textit{School of Mathematics} \\
\textit{University of Edinburgh}\\
Edinburgh, United Kingdom \\
anjos@stanfordalumni.org}
\and
\IEEEauthorblockN{James R. Cruise}
\IEEEauthorblockA{\textit{Riverlane Research} \\
Cambridge, United Kingdom \\
james.cruise@riverlane.io}
\and
\IEEEauthorblockN{Albert Sol\`a Vilalta}
\IEEEauthorblockA{\textit{Maxwell Institute for Mathematical Sciences} \\
\textit{School of Mathematics} \\
\textit{University of Edinburgh}\\
Edinburgh, United Kingdom \\
albert.sola@ed.ac.uk}
}

\maketitle

\begin{abstract}
Energy storage and demand-side response will play an increasingly important role in the future electricity system. We extend previous results on a single energy storage unit to the management of two energy storage units cooperating for the purpose of price arbitrage. We consider a deterministic dynamic programming model for the cooperative problem, which accounts for market impact. We develop the Lagrangian theory and present a new algorithm to identify pairs of strategies. While we are not able to prove that the algorithm provides optimal strategies, we give strong numerical evidence in favour of it. Furthermore, the Lagrangian approach makes it possible to identify decision and forecast horizons, the latter being a time beyond which it is not necessary to look in order to determine the present optimal action. In practice, this allows for real-time reoptimization, with both horizons being of the order of days.
\end{abstract}


\section{Introduction}

Power systems around the world are facing many challenges due to decarbonization. Under the revised Climate Change Act 2008 \cite{ClimateChangeAct08}, the UK set ambitious plans to reduce carbon emissions $100 \%$ (net zero) by 2050 compared to 1990 levels. This requires fundamental changes in many sectors, including the power system, which are already taking place. We are interested in the role grid-scale electric energy storage can play in this transformation.

The deregulation of energy markets, driven in the UK by the Electricity Act 1989 \cite{ElectricityAct89}, created a more favourable environment for energy storage to enter the power system \cite{Graves99}. For example, the price-taker and price-maker cases were studied in \cite{Sioshansi09}, which also analysed the arbitrage value in the former case using PJM price data. They conclude there was a significant increase in the arbitrage value of energy storage in the late 2000s, which led to an increasing interest in energy storage. Moreover, \cite{Denholm10} describes four other reasons for this increasing interest: advances in storage technologies, increase in fossil fuel prices, challenges to sitting new transmission and distribution facilities, and opportunities for storage with variable renewable generation.

A possible way to capture the value of energy storage to the system is through price arbitrage. How should one optimally manage a fleet of electric energy storage units with market impact to maximize profit by buying electricity when it is cheap and selling it when it is expensive? This was studied for a single storage unit in \cite{Cruise19Lagrangian}. The aim of this work is to study the problem for two storage units. We recognize that not all the value of storage can be captured by price arbitrage, since it can provide other services such as operating reserves, firm capacity, network reinforcement deferral, black-start support, power quality and stability, and aid in the integration of renewables \cite{Cruise18Buffering}, \cite{Denholm10}. Nevertheless, arbitrage has already been used to approximate the value of storage \cite{Cruise19Lagrangian}, \cite{Sioshansi09}.

There are essentially two ways in which energy storage units can interact: cooperation or competition. In the former, the objective is to maximize the joint profit made by the units, while in the latter, the objective of each unit is to maximize its own profit. See \cite{Biegel13}, \cite{Cruise18Competition}, \cite{Schill11}, \cite{Wang14} and \cite{Zhang13} for examples of both problems in the context of energy storage or demand-side response. Our work focuses on the cooperative problem.

Market prices can be modelled as deterministic or stochastic. For stochastic approaches, see \cite{Baeuerle16}, \cite{Felix13} and \cite{Secomandi10}. \cite{Secomandi15} argues that assumed probability distributions calibrated to data might be incorrect. Therefore, a deterministic approach that reoptimizes once more accurate forecasts become available might avoid this problem \cite{Barbry19}, \cite{Cruise19Lagrangian}, \cite{Lai10}, \cite{Secomandi10}, \cite{Secomandi15_2}, \cite{Shafiee16}, \cite{Wu12}. We assume that prices are deterministic.

We consider a discrete time mathematical model. We assume a storage unit is characterized by its energy capacity $E$ and power rate constraint $P$. The crucial assumption, which makes the problem interesting, is that units are large enough to have market impact, leading to convex costs. Therefore, we are facing a nonlinear convex dynamic programming problem, for which we take a Lagrangian approach. 

Our model is technology agnostic, and therefore can account for any storage technology, or, more broadly, flexibility services, including demand-side response. At present, pumped-hydro is essentially the only grid-scale storage technology that has market impact. However, the fast decreasing prices of many storage technologies \cite{WorldEnergyOutlook18} might mean that they will start having market impact in the near future.

This problem could be seen as an instance of the \emph{warehouse} or \emph{wheat trading problem}, introduced in \cite{Cahn48}, which reads

\textit{Given a warehouse with fixed capacity and an initial stock of a certain product, which is subject to known seasonal price and cost variations, what is the optimal pattern of purchasing (or production), storage and sales?}

The two key differences between our problem and the literature on the classical warehouse problem \cite{Bellman56}, \cite{Charnes55}, \cite{Dreyfus57} are the power rate constraints and the convexity introduced by the market impact assumption. In the classical case, costs are linear as a function of the amount bought, or concave to account for economics of scale. Furthermore, there is relevant literature in scheduling hydroelectric generation \cite{Ambec03}, \cite{Rangel08}, \cite{Steeger14}, which resembles the discharging process of energy storage.

A novel feature of the results presented in \cite{Cruise19Lagrangian} for the single storage unit problem is the identification of forecast horizons, a time beyond which it is not necessary to look to determine the present optimal action. They appear as a consequence of the capacity constraints. Our numerical experiments in Section \ref{SEC:NumericalImplementationAndExamples} suggest that they also exist for our problem and coincide with those of the storage unit with largest $E/P$ ratio.

The difference between our approach and a standard rolling horizon approach \cite{Sethi91} is that horizons are not specified in advance, but identified in the process of finding a solution. This means that, once horizons are found, we are certain that no more future price information is needed to determine the present action. See \cite{Chand02} for more on horizons.

This work is divided in five sections. In Section \ref{SEC:theModel}, we introduce the mathematical model, and in Section \ref{SEC:LagrangianSufficiency} we present Lagrangian sufficient conditions for a solution to be optimal. Based on them, in Section \ref{SEC:algorithm}, we introduce a new algorithm to find an optimal solution. We do not give a mathematical proof that it provides an optimal solution, but our numerical experiments in Section \ref{SEC:NumericalImplementationAndExamples} suggest it.

\section{The Model} \label{SEC:theModel}

An \emph{energy storage unit} or \emph{unit} $(E, P)$ is determined by its energy capacity $E$ and power rate $P$. The \emph{energy capacity} is the maximum amount of energy that can be stored in the unit, typically in MWh. The \emph{power rate} is the maximum rate at which energy can be charged and discharged, typically in MW. Our modelling assumptions are consistent with \cite{Cruise19Lagrangian}, \cite{Graves99} and \cite{Sioshansi09}, albeit we do not consider efficiencies nor different charging and discharging power rates.

Let $(E_1, P_1), (E_2, P_2)$ be two units. Consider discrete time steps $1, 2, \dots, T$. We assume electricity prices are deterministic. To model them, we introduce \emph{cost functions}
\begin{equation}
	C_t: [- (P_1 + P_2), P_1 + P_2] \longrightarrow \mathbb{R}
\end{equation} 
for every time step $t = 1, 2, \dots, T$. Since we are interested in the cooperative problem, cost functions are functions of the total amount of energy charged or discharged in a time step. For positive $x$, $C_t(x)$ is the cost of buying an amount $x$ of electricity at time $t$, whereas for negative $x$, $C_t(x)$ is the negative of the reward of selling an amount $x$ of electricity at time $t$.

We assume that $C_t$ is monotonically increasing, strictly convex and $C_t(0) = 0$ for $t = 1, 2, \dots, T$. The motivation to consider increasing cost functions is clear; the more energy we buy, the more we have to pay. Convexity accounts for the market impact assumption. The more energy we buy (resp. sell), the higher (resp. lower) the demand is, and therefore the higher (resp. lower) the price per unit of electricity is. It is possible to relax the strict convexity assumption to just convexity, with similar ideas to those in \cite{Cruise19Lagrangian}, but we shall not treat this case here. The last assumption corresponds to the fact that doing nothing should not have a cost nor a reward.

A \emph{strategy} is a pair of vectors $S = (S_1, S_2) \in \mathbb{R}^{T + 1} \times \mathbb{R}^{T + 1}$, where $S_{j, t}$ denotes the level of charge of unit $j$ at time $t$ for $j = 1, 2$ and $t = 0, 1, \dots, T$. We fix the initial and final levels of charge to be $\bar{S}_{j, 0}$ and $\bar{S}_{j, T}$ for $j = 1, 2$. This last fixing is essential, otherwise units would be as empty as possible at time $T$.

Given a strategy $S$, we define the increments $(x_1(S), x_2(S)) \in \mathbb{R}^T \times \mathbb{R}^T$ associated to it by
\begin{equation} \label{EQ:increments}
	x_{j, t}(S) := S_{j, t} - S_{j, t - 1}
\end{equation}
for $j = 1, 2$ and $t = 1, 2, \dots, T$. Our aim is to maximize the profit made by the two units subject to capacity and rate constraints. This leads to the following optimization problem:

$\boldsymbol{\mathcal{P}}$: Minimize
\begin{equation} \label{EQ:objectiveFunction}
	\sum_{t = 1}^T C_t(x_{1, t}(S) + x_{2, t}(S))
\end{equation}
amongst $S \in \mathbb{R}^{T + 1} \times \mathbb{R}^{T + 1}$ subject to capacity constraints
\begin{equation} \label{EQ:capacityConstraints}
	\begin{split}
	& S_{j, 0} = \bar{S}_{j, 0} \\
	& 0 \leq S_{j, t} \leq E_j, \hspace{2mm} 1 \leq t \leq T - 1 \\
	& S_{j, T} = \bar{S}_{j, T} \\
	\end{split}
\end{equation}
and rate constraints
\begin{equation} \label{EQ:rateConstraints}
	-P_j \leq x_{j, t}(S_j) \leq P_j, \hspace{2mm} 1 \leq t \leq T - 1
\end{equation}
for $j = 1, 2$.

We have taken this problem from \cite{Cruise18Competition}, but we recognize that the modelling is very similar to other previous works, for instance \cite{Graves99} and \cite{Sioshansi09}. A strategy $S$ satisfying both capacity \eqref{EQ:capacityConstraints} and rate \eqref{EQ:rateConstraints} constraints is called a \emph{feasible strategy}.

There are at least two cases where problem $\boldsymbol{\mathcal{P}}$ reduces to single unit problems. First, if the cost functions are linear, the objective function \eqref{EQ:objectiveFunction} becomes the sum of a function of $S_1$ and a function of $S_2$, since $x_{j, t}(S)$ is a function of $S_j$ only, see \eqref{EQ:increments}. Therefore, solving $\boldsymbol{\mathcal{P}}$ is equivalent to solving two single unit problems. 

Second, if the units satisfy
\begin{equation}
	\frac{E_1}{P_1} = \frac{E_2}{P_2},
\end{equation}
then we can consider the single unit problem with $(E, P) = (E_1 + E_2, P_1 + P_2)$. Define $\lambda \in [0, 1]$ to satisfy
\begin{equation}
	E_1 = \lambda E = \lambda (E_1 + E_2).
\end{equation}
Given any feasible strategy $R$ for the single unit problem, consider the strategy $S = (S_1, S_2) = (\lambda R, (1 - \lambda) R)$, which is feasible for problem $\boldsymbol{\mathcal{P}}$ and has the same cost as $R$.

\section{Lagrangian sufficiency} \label{SEC:LagrangianSufficiency}

We now present the Lagrangian sufficiency theorem for problem $\boldsymbol{\mathcal{P}}$. It introduces $\mu_1^*, \mu_2^* \in \mathbb{R}^T$, which are essentially vectors of cumulative Lagrange multipliers. They play a crucial role in our algorithm. The theorem is strongly based on Theorem 1 from \cite{Cruise19Lagrangian}, where the corresponding cumulative Lagrange multipliers are described in detail.

\begin{theorem} \label{THM:sufficientConditions2Stores}
Assume there exist pairs $(S_1^*, \mu_1^*), (S_2^*, \mu_2^*) \in \mathbb{R}^{T + 1} \times \mathbb{R}^T$ satisfying the following conditions:
\begin{enumerate}
	\item[(i)] $S^* = (S_1^*, S_2^*)$ is a feasible strategy for problem $\boldsymbol{\mathcal{P}}$.
	\item[(ii)] For $t = 1, 2, \dots, T$, $(x_{1, t}(S^*), x_{2, t}(S^*))$ minimizes
	\begin{equation} \label{EQ:minimisationEquation2Stores}
		C_t(x_1 + x_2) - \mu_{1, t}^* x_1 - \mu_{2, t}^* x_2
	\end{equation}
	amongst $(x_1, x_2) \in [-P_1, P_1] \times [-P_2, P_2]$.
	\item[(iii)] $(S_1^*, \mu_1^*), (S_2^*, \mu_2^*)$ satisfy the \emph{complementary slackness conditions}
	\begin{equation*}
		\left \{
			\begin{matrix}
				\mu_{j,t + 1}^* = \mu_{j, t}^* & \text{if } & 0 < S_{j, t}^* < E_j \\
				\mu_{j, t + 1}^* \leq \mu_{j, t}^* & \text{if } & S_{j, t}^* = 0 \\
				\mu_{j, t + 1}^* \geq \mu_{j, t}^* & \text{if } & S_{j, t}^* = E_j \\
			\end{matrix}
		\right .
	\end{equation*}
	for $j = 1, 2$ and $t = 1, 2, \dots, T - 1$.
\end{enumerate}
Then $(S_1^*, S_2^*)$ solves problem $\boldsymbol{\mathcal{P}}$.
\end{theorem}

\begin{proof}
Let $(S_1, S_2)$ be any feasible strategy for problem $\boldsymbol{\mathcal{P}}$. From (ii), we get
\begin{equation*}
	\begin{split}
		& \sum_{t = 1}^T C_t(x_{1, t}(S^*) + x_{2, t}(S^*)) - \mu_{1, t}^* x_{1, t}(S^*) - \mu_{2, t}^* x_{2, t}(S^*) \\
		& \leq \sum_{t = 1}^T C_t(x_{1, t}(S) + x_{2, t}(S)) - \mu_{1, t}^* x_{1, t}(S) - \mu_{2, t}^* x_{2, t}(S). \\
	\end{split}
\end{equation*}
Rearranging and using the capacity constraints $S_{j, 0} = \bar{S}_{j, 0}, S_{j, T} = \bar{S}_{j, T}$ for $j = 1, 2$, we get
\begin{equation*}
	\begin{split}
	& \sum_{t = 1}^T C_t(x_{1, t}(S^*) + x_{2, t}(S^*)) - C_t(x_{1, t}(S) + x_{2, t}(S)) \\
	& \leq \sum_{t = 1}^T \mu_{1, t}^*(x_{1, t}(S^*) - x_{1, t}(S)) + \mu_{2, t}^*(x_{2, t}(S^*) - x_{2, t}(S)) \\
	& = \sum_{t = 1}^T \big [\mu_{1, t}^*(S_{1, t}^* - S_{1, t - 1}^* - S_{1, t} + S_{1, t- 1}) \\ 
	& \hspace{8mm} + \mu_{2, t}^* (S_{2, t}^* - S_{2, t - 1}^* - S_{2, t} + S_{2, t -1}) \big ] \\
	& = \sum_{t = 1}^{T - 1} \big [(\mu_{1, t}^* - \mu_{1, t + 1}^*)(S_{1, t}^* - S_{1, t}) \\
	& \hspace{8mm} + (\mu_{2, t}^* - \mu_{2, t + 1}^*)(S_{2, t}^* - S_{2, t}) \big ] \leq 0 \\
	\end{split}
\end{equation*}
	where the last inequality follows from (iii).
\end{proof}

\section{Algorithm} \label{SEC:algorithm}

We now introduce a new algorithm to solve problem $\boldsymbol{\mathcal{P}}$, via Theorem \ref{THM:sufficientConditions2Stores}. It uses the single unit algorithm from \cite{Cruise19Lagrangian} in intermediate steps. The algorithm from \cite{Cruise19Lagrangian} performs a search on the parameter $\mu \in \mathbb{R}^T$. Given strictly convex cost functions $C_t$ for $t = 1, 2, \dots, T$, it finds a parameter $\mu^* \in \mathbb{R}^T$ and a vector $x(\mu^*) \in [-P, P]^T$, whose components $x_t(\mu) = x_t(\mu_t^*)$ are both the unique minimizer of
\begin{equation}
	C_t(x) - \mu_t^* x, \hspace{2mm} -P \leq x \leq P
\end{equation}
and an optimal action at time $t$. To find them, the functions
\begin{equation}
	\begin{matrix}
		\hat{x}_t: & \mathbb{R} & \longrightarrow & [-P, P] \\
		& \mu & \longmapsto & \hat{x}_t(\mu)
	\end{matrix}
\end{equation}
need to be monotonically increasing and surjective for $t = 1, 2, \dots, T$, where $\hat{x}_t(\mu)$ denotes the unique minimizer of
\begin{equation}
	C_t(x) - \mu x, \hspace{2mm} -P \leq x \leq P.
\end{equation}

Coming back to problem $\boldsymbol{\mathcal{P}}$, for any $(\mu_1, \mu_2) \in \mathbb{R}^2$, define the function $C_t^{\mu_1, \mu_2}$ by
\begin{equation}
	C_t^{\mu_1, \mu_2}(x_1, x_2) := C_t(x_1 + x_2) - \mu_1 x_1 - \mu_2 x_2,
\end{equation}
where $x_j \in [-P_j, P_j]$ for $j = 1, 2$. As in the single unit problem, we would like to associate a unique minimizer of $C_t^{\mu_1, \mu_2}$ to each parameter $(\mu_1, \mu_2)$ in such a way that every minimizer of $C_t^{\mu_1, \mu_2}$ has a parameter associated to it. By considering only increments obtained this way, condition (ii) of Theorem \ref{THM:sufficientConditions2Stores} will be automatically satisfied.

If $\mu_1 \neq \mu_2$, then the cost function $C_t$ being strictly convex implies that there is a unique minimizer of $C_t^{\mu_1, \mu_2}$. However, if $\mu_1 = \mu_2$, then $C_t^{\mu_1, \mu_1}$ becomes a function of $x_1 + x_2$ and therefore constant on the segments $r_{\Gamma}$ defined by
\begin{equation}
	\{(x_1, x_2) \in [-P_1, P_1] \times [- P_2, P_2] \hspace{2mm} | \hspace{2mm} x_1 + x_2 = \Gamma \}
\end{equation}
for all $\Gamma \in \mathbb{R}$. Since $C_t$ is strictly convex, there exists a unique segment that minimizes $C_t^{\mu_1, \mu_1}$, which we call $r_{\mu_1}$. To overcome this difficulty, we enlarge the space of parameters $(\mu_1, \mu_2) \in \mathbb{R}^2$. Let $\kappa_2 \in [0, 1]$ and define
\begin{equation} \label{EQ:enlargedParameter}
	\nu = (\nu_1, \nu_2) := (\mu_1, (\mu_2, \kappa_2)) \in A,
\end{equation}
where $A := \mathbb{R} \times (\mathbb{R} \times [0, 1])$. In what follows, it will be understood that $\mu_1$, $\mu_2$ and $\kappa_2$ are the components of the enlarged parameter $\nu$, as in \eqref{EQ:enlargedParameter}, unless stated otherwise.

We define an order relation on the second component $\nu_2$ of enlarged parameters by $\nu_2 = (\mu_2, \kappa_2) < \nu_2' = (\mu_2', \kappa_2')$ if
\begin{equation} \label{EQ:orderEnlargedParam}
	 \mu_2 < \mu_2' \text{ or } (\mu_2 = \mu_2' \text{ and } \kappa_2 < \kappa_2').
\end{equation}
Given $\nu \in A$, assign a minimizer $(\hat{x}_{1, t}(\nu), \hat{x}_{2, t}(\nu))$ of $C_t^{\mu_1, \mu_2}$ to it as follows. If $\mu_1 \neq \mu_2$, let it be the unique minimizer of $C_t^{\mu_1, \mu_2}$, irrespective of the value of $\kappa_2$. If $\mu_1 = \mu_2$, let $x_{j, t}^-(\mu_1)$ and $x_{j, t}^+(\mu_1)$ be the minimum and maximum values of the $x_j$-component on the minimizing segment $r_{\mu_1}$. Define
\begin{equation}
	\begin{split}
	\hat{x}_{1, t}(\nu) & := (1 - \kappa_2) x_{1, t}^+(\mu_1) + \kappa_2 x_{1, t}^-(\mu_1) \\
	\hat{x}_{2, t}(\nu) & := (1 - \kappa_2) x_{2, t}^-(\mu_1) + \kappa_2 x_{2, t}^+(\mu_1)
	\end{split}
\end{equation}
for $t = 1, 2, \dots, T$. In other words, we use the extra parameter $\kappa_2$ to parametrize the minimizing line $r_{\mu_1}$. We can now define the strategy $(S_1(\nu), S_2(\nu))$ associated to $\nu \in A$ recursively, by following these increments, i.e.,
\begin{equation}
	S_{j, 0}(\nu) := \bar{S}_{j, 0}, \hspace{2mm} S_{j, t}(\nu) := S_{j, t - 1}(\nu) + \hat{x}_{j, t}(\nu)
\end{equation}
for $j = 1, 2$ and $t = 1, 2, \dots, T$. In other words,
\begin{equation} \label{EQ:strategiesAssociatedToNuv2}
	S_{j, t}(\nu) := \bar{S}_{j, 0} + \sum_{s = 1}^t \hat{x}_{j, s}(\nu)
\end{equation}
for $j = 1, 2$ and $t = 0, 1, \dots, T$.

Our aim is now to choose a vector of enlarged parameters $(\nu_1, \dots, \nu_T) \in A^T$ such that conditions (i) and (iii) of Theorem \ref{THM:sufficientConditions2Stores} are also satisfied. The crucial observation is that fixing $\mu_1$ (or, analogously, $\mu_2$), brings us to a situation where we can apply the single unit algorithm \cite{Cruise19Lagrangian} to unit 2. Indeed, for fixed $\mu_1 = \bar{\mu}_1$, for every enlarged parameter $\nu = (\bar{\mu}_1, (\mu_2, \kappa_2))$, there exists a unique minimizer $(\hat{x}_{1, t}(\nu), \hat{x}_{2, t}(\nu))$ of $C_t^{\bar{\mu}_1, \mu_2}$ associated to it. Furthermore, the functions
\begin{equation}
	\begin{matrix}
		\hspace{5mm} \hat{x}_{2, t}(\bar{\mu}_1, \cdot): & \mathbb{R} \times [0, 1] & \longrightarrow & [-P_2, P_2] \\
		& (\mu_2, \kappa_2) & \longmapsto & \hat{x}_{2, t}(\bar{\mu}_1, (\mu_2, \kappa_2)) \\
	\end{matrix}
\end{equation}
are monotonically increasing and surjective for $t = 1, 2, \dots T$. Note that to guarantee monotonicity, definition \eqref{EQ:orderEnlargedParam} is essential if $\mu_2 = \bar{\mu}_1$. Therefore, we can apply the single unit algorithm to unit 2 with $\mu_1 = \bar{\mu}_1$ fixed, and obtain a vector of parameters $\nu_2$, which we denote by $M_2(\bar{\mu}_1) = (M_{2, 1}(\bar{\mu}_1), \dots, M_{2, T}(\bar{\mu}_1)) \in (\mathbb{R} \times [0, 1])^T$ together with the associated strategies $(S_1(\bar{\mu}_1, M_2(\bar{\mu}_1)), S_2(\bar{\mu}_1, M_2(\bar{\mu}_1)))$, whose components are given by
\begin{equation} \label{EQ:strategyAssociatedMu1}
	S_{j, t}(\bar{\mu}_1, M_2(\bar{\mu}_1)) := \bar{S}_{j, 0} + \sum_{s = 1}^t \hat{x}_{j, s}(\bar{\mu}_1, M_{2, s}(\bar{\mu}_1))
\end{equation}
for $j = 1, 2$ and $t = 0, 1, \dots, T$. Note the difference between \eqref{EQ:strategiesAssociatedToNuv2} and \eqref{EQ:strategyAssociatedMu1}. In \eqref{EQ:strategyAssociatedMu1}, $\nu_2$ changes over time, while in \eqref{EQ:strategiesAssociatedToNuv2} it was fixed. The obtained strategy is feasible for unit 2, since it is a direct application of the single unit algorithm, whereas unit 1 might still break its capacity constraints.

By applying the single unit algorithm to unit 2 with $\mu_1 = \bar{\mu}_1$ fixed, we reduce the dimension of the vector of enlarged parameters needed to obtain strategies associated to a parameter. This is clear from \eqref{EQ:strategyAssociatedMu1}, which depends only on $\bar{\mu}_1$. Therefore, we could continue by applying the single unit algorithm to unit 1, doing a linear search in $\mu_1$. For every $\bar{\mu}_1$, apply the process described above to obtain $M_2(\bar{\mu}_1)$ and the associated strategy \eqref{EQ:strategyAssociatedMu1}. It is difficult to show that the functions
\begin{equation} \label{EQ:incrementsStoreOneMonotone}
	\begin{matrix}
		\hspace{5mm} \hat{x}_{1, t}: & \mathbb{R} & \longrightarrow & [-P_2, P_2] \\
		& \mu_1 & \longmapsto & \hat{x}_{1, t}(\mu_1, M_{2, t}(\mu_1)) \\
	\end{matrix}
\end{equation}
are monotonically increasing and surjective and therefore there is no guarantee that we can do this. The main difficulty is that $\nu_2 = M_{2, t}(\mu_1)$ is not fixed, but varies as we change $\mu_1$. Our numerical experiments in Section \ref{SEC:NumericalImplementationAndExamples} suggest that it is possible to apply the single unit algorithm to unit 1.

In order to satisfy condition (iii) of Theorem \ref{THM:sufficientConditions2Stores}, we need to make sure that $\mu_{j, t}$ changes only when unit $j$ is empty or full. The proposed algorithm does not change $\mu_{1, t}$ unless unit 1 is empty or full, since $\mu_1$ is chosen in the outer application of the single unit algorithm. Nevertheless, a change in $\mu_{1, t}$, when unit 1 is full (resp. empty), might change $M_{2, t}(\mu_1)$, which would make unit 2 not satisfy condition (iii) of Theorem \ref{THM:sufficientConditions2Stores}, unless unit 2 is also full (resp. empty). Therefore, can any unit play the role of unit 1, or we need to make a particular choice?

The answer to this question is that unit 1 needs to satisfy
\begin{equation} \label{EQ:orderInStores}
	\frac{E_1}{P_1} > \frac{E_2}{P_2}.
\end{equation}
The rationale behind this choice is that the larger the $E/P$ ratio is, the longer it takes a unit to complete a charging/discharging cycle \cite{Cruise19Lagrangian}. Therefore, for unit 2 to be empty (resp. full) whenever unit 1 is also empty (resp. full), unit 2 needs to have a shorter cycle. There is no guarantee that this will happen, but the opposite choice would make it almost impossible. We explore this numerically in Section \ref{SEC:NumericalImplementationAndExamples}, linking it to decision and forecast horizons for problem $\boldsymbol{\mathcal{P}}$.

The algorithm may be summarized as follows:

\begin{enumerate}
	\item[1.] Sort the units by $E/P$-ratios, according to \eqref{EQ:orderInStores}.
	\item [2.] For every $\mu_1 \in \mathbb{R}$, obtain $M_2(\mu_1) = (M_{2, 1}(\mu_1), \dots, M_{2, T}(\mu_1)) \in (\mathbb{R} \times [0, 1])^T$ and associated strategies as defined by \eqref{EQ:strategyAssociatedMu1}.
	\item [3.] Find $\mu_1^* \in \mathbb{R}$ that makes unit 1 follow the single unit algorithm, and identify the corresponding decision and forecast horizons $\tau_1$ and $\bar{\tau}_1$. Define the strategy $S^* = (S_1^*, S_2^*)$ until the decision horizon $\tau_1$ to be
		\begin{equation}
			S_{j, t}^* = S_{j, t}(\mu_1^*, M_2(\mu_1^*))
		\end{equation}
		for $j = 1, 2$ and $t = 0, 1, \dots, \tau_1$.
	\item[4.] At this point, from the single unit algorithm \cite{Cruise19Lagrangian}, we know that
\begin{equation}
	S_{1, \tau_1}^* = 0 \text{ or } S_{1, \tau_1}^* = E_1.
\end{equation}	
	 Check whether unit 2 is in the same state as unit 1, i.e., empty (resp. full) if unit 1 is empty (resp. full).
\end{enumerate}
If $\tau_1 = T$, stop. Otherwise, go back to 2. with $\bar{S}_{j, 0} = S_{j, \tau_1}^*$.

\section{Numerical Implementation and Example} \label{SEC:NumericalImplementationAndExamples}

\begin{figure*}[h!]
	\centering
	\includegraphics[width= \textwidth]{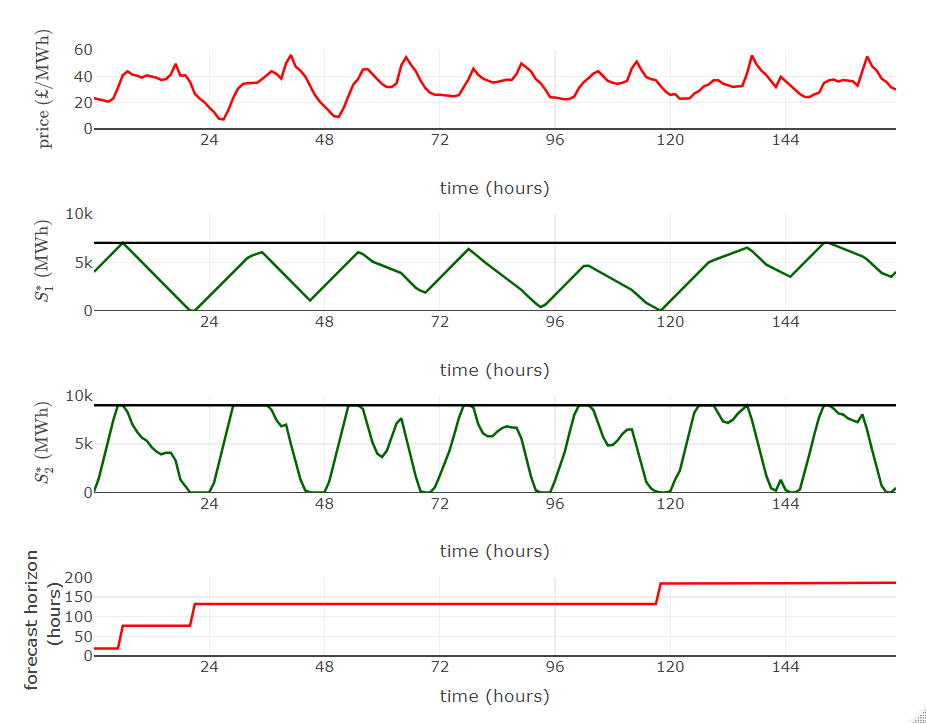}
	\caption{Control of two storage units provided by our algorithm.}
	\label{FIG:twoStoresOneWeek}
\end{figure*}

We present a numerical example using UK day-ahead hourly electricity prices of January 2020 \cite{NordPool}. Price data is given by a vector $p = (p_1, \dots, p_T) \in \mathbb{R}^T$, where $p_t$ is the reference price at time $t$. We assume that cost functions are given by
\begin{equation} \label{EQ:costFunctionsFromPrices}
	C_t(\xi) = (p_t + \lambda p_t \xi) \xi, \hspace{2mm} \xi \in [-(P_1 + P_2), P_1 + P_2]
\end{equation}
for any $t = 1, 2, \dots, T$. Here $\xi$ represents the combined action of the units and $\lambda > 0$ is the market impact factor. Market impact is modelled as in \cite{Cruise19Lagrangian}, which is consistent with existing energy economics literature \cite{Sioshansi10}, \cite{Sioshansi14}.

Consider two energy storage units $(E_1, P_1) = (7000, 500)$ and $(E_2, P_2) = (9000, 2000)$, which correspond approximately to the energy capacities in MWh and power rates in MW of Cruachan and Dinorwig pumped-storage power stations, respectively \cite{REA16}. Since we are considering hourly time periods, the maximum amount of energy unit $j$ can be charged or discharged in one time period is $P_j$ for $j = 1, 2$. Note that, although Dinorwig has larger energy capacity and power rate, it plays the role of unit 2 due to its smaller $E/P$ ratio. We consider a market impact factor $\lambda = 5 \times 10^{-5}$, which corresponds to that considered in \cite{Cruise19Lagrangian} after rescaling the data.

We present the numerical results in Figure \ref{FIG:twoStoresOneWeek}. We focus on the week starting on Monday, January 13, the first time period displayed being the first hour that day. The algorithm was run for the whole month of January to have realistic initial levels of charge and capture weekday-weekend variations.

Figure \ref{FIG:twoStoresOneWeek} contains four plots. The $x$-axis is always time, in hours. From top to bottom, the first plot contains the prices $p$ \cite{NordPool}. The second and third plots contain the components of the strategy $S^* = (S_1^*, S_2^*)$ obtained by the algorithm. Finally, the bottom plot contains the forecast horizon, in hours.

The first observation is that the algorithm produces a feasible strategy. This is strong evidence suggesting that \eqref{EQ:incrementsStoreOneMonotone} is indeed monotonically increasing and surjective. Furthermore, we can also see that whenever unit 1 is full (resp. empty), unit 2 is also full (resp. empty). This is strong evidence suggesting that the obtained strategy $S^*$ and $\mu$-values satisfy condition (iii) of Theorem \ref{THM:sufficientConditions2Stores}. It also suggests that decision and forecast horizons for problem $\boldsymbol{\mathcal{P}}$ should be those of unit 1. Moreover, since unit 2 gets empty (resp. full) without unit 1 being empty (resp. full), as in the period between $t = 24$ and $t = 48$, it is clear that the choice in \eqref{EQ:orderInStores} is essential.

We can check a posteriori if the assumptions of Theorem \ref{THM:sufficientConditions2Stores} are satisfied, which is the case for the cost functions considered, meaning that the obtained strategy is optimal. An open question is whether the algorithm produces a feasible strategy for any cost functions satisfying the assumptions in Section \ref{SEC:theModel}, and, if so, whether it is optimal.

We also observe that there are periods of cross charging, where one unit charges while the other discharges, as it can be seen before the evening peaks, at around $t = 14, 39, 62, 84, 109, 134, 159$. We expect the frequency of these periods to be reduced by introducing efficiencies into the model, but might still be valuable in certain circumstances.

The bottom plot on forecast horizons is interesting. Since a group of actions are determined simultaneously, namely those actions to be taken before the next decision horizon, the forecast horizon at those times stays constant. The forecast horizon increases rapidly until $t = 20$, to then stay constant until $t = 118$, as unit 1 does not get empty or full between those times. The difference in length of the periods where the forecast horizon stays constant highlights that they are not set in advance, as in a standard rolling horizon approach \cite{Sethi91}. Furthermore, the look-ahead time, i.e., the difference between the forecast horizon at time $t$ and $t$ is of the order of days.

Finally, we observe that for the given prices $p$, the units make most of their profit from intraday price variations. This is clear from the strong daily cycles of both units, which are consistent with the price cycles. Unit 1 also takes advantage of cheaper weekend prices, by starting full on Monday morning, $t = 6$, and being empty by Friday night, $t = 118$.

\section{Conclusions}

We have presented a model, developed the associated Lagrangian theory, and introduced a new algorithm to solve the cooperative two storage unit problem with market impact. We have not given a mathematical proof that the algorithm provides an optimal solution, but the numerical experiments give strong evidence in favour of it. They also suggest that decision and forecast horizons exist for this problem and are precisely those of the storage unit with largest $E/P$ ratio.

In a real-world application, prices are uncertain and would need to be forecasted. Furthermore, storge units might need to commit to their actions some time in advance. The existence of decision and forecast horizons makes it possible for the storage units to commit until the next decision horizon and reoptimize their actions after that time with more accurate price forecasts. This makes this approach suitable to deal with uncertainty.

Future work will consider the $n$ storage unit problem and deal with storage efficiencies.

\section*{Acknowledgments}

The authors wish to thank Fraser Daly, Chris Dent, Jean Lasserre, Seva Shneer and Stan Zachary for very helpful discussions and the anonymous reviewers for many useful comments and suggestions.

\bibliography{references}
\bibliographystyle{plain}

\end{document}